\documentclass[11pt]{amsart}
\usepackage[active]{srcltx}
\usepackage{amsmath,amssymb}
\usepackage{latexsym}
\usepackage{color}
\usepackage{dsfont}
\usepackage[normalem]{ulem}
\usepackage{epsfig}
\usepackage{pgf,pgfarrows,pgfnodes,pgfautomata,pgfheaps}
\newtheorem{teo}{Theorem}[section]
\newtheorem{lema}[teo]{Lemma}

\newtheorem{cor}[teo]{Corollary}

\usepackage{epstopdf}

\begin{document}

\title[quasilinear heat equation with a localized reaction]{Grow-up for a quasilinear heat equation with a localized reaction}

\author{R. Ferreira and A. de Pablo}

\address{Ra\'{u}l Ferreira
\hfill\break\indent  Departamento de Matem\'{a}ticas,
\hfill\break\indent U.~Complutense de Madrid,
\hfill\break\indent
 28040 Madrid, Spain.
\hfill\break\indent  e-mail: {\tt raul$_-$ferreira@ucm.es}}

\address{Arturo de Pablo
\hfill\break\indent  Departamento de Matem\'{a}ticas,
\hfill\break\indent U. Carlos III de Madrid,
\hfill\break\indent
 28911 Legan\'{e}s, Spain.
\hfill\break\indent  e-mail: {\tt arturop@math.uc3m.es}}

\maketitle

\

\begin{abstract}
We study the behaviour of  global solutions to the quasilinear heat equation  with a reaction localized
$$
u_t=(u^m)_{xx}+a(x) u^p,
$$
$m, p>0$ and $a(x)$ being the characteristic function of an interval. we prove that there exists $p_0=\max\{1,\frac{m+1}2\}$ such that all global solution are bounded if $p>p_0$, while for $p\le p_0$ all the solution are global and unbounded. In the last case, we  prove that if $p<m$ the grow-up rate is different  to the one obtained when $a(x)\equiv1$, while if $p>m$ the grow-up rate coincides with  that rate, but only inside the support of $a$; outside the interval the rate is smaller.
\end{abstract}

%%%%%%%%%%%%%%%%%%%%%%%%%%%%%%%%%%%%%%%%%%%%%%%%

\section{Introduction}

\label{sect-introduction} \setcounter{equation}{0}

We consider non-negative solutions to the following problem
\begin{equation}\label{eq.principal}
\left\{
\begin{array}{ll}
u_t=(u^m)_{xx}+a(x) u^p\qquad & x\in \mathbb R,\;t>0,\\
u(x,0)=u_0(x)
\end{array}\right.
\end{equation}
with $m,\,p>0$. The initial datum is a continuous, nonnegative and nontrivial  function $u_0\in L^1(\mathbb{R})\cap L^\infty(\mathbb{R})$. The reaction coefficient is a characteristic function of an interval, which without loss of generality can be assumed symmetric, $I=I_L=(-L,L)$, $a(x)=\ell\,\mathds{1}_{I}(x)$. The parameters $L,\,\ell>0$ can be eliminated by a simple rescaling when $p\ne m$. On the other hand, if $p=m$, putting $\ell=1$ means changing $L$ into $L/\sqrt\ell$. We therefore assume $\ell=1$, $0<L<\infty$.

The existence of a solution to problem~\eqref{eq.principal}, local in time, can be easily achieved. But uniqueness is subtle when $p<1$, due to the  non-Lipschitz character of the reaction,  see~\cite{AguirreEscobedo,dePabloVazquez}. Actually,  initial data that  vanish somewhere in $I$, where the reaction applies, can cause nonuniqueness, see Appendix. We therefore assume in the sequel that $u_0$ is strictly positive in the closed interval $\overline{I}$.  Finally, if $T$ is the maximal time of existence of the unique solution then the solution is bounded in $\mathbb{R}^N\times[0,t]$ for every $t<T$.

Problems like \eqref{eq.principal} are studied mainly when $p>1$ and in the context of blow-up, i.e. when  $T$ is finite, and in that case
\begin{equation}\label{bup}
\|u(\cdot,t)\|_\infty\to \infty \quad \mbox{as } t\to T.
\end{equation}
To begin with, the case with global reaction, $L=\infty$, has been described
by Fujita in the semilinear case $m=1$ in the the seminal work \cite{Fujita}, even in dimension $N\ge1$, where the equation is
$$
u_t=\Delta u+u^p.
$$
It is proved in that paper that there exist two exponents, the \emph{global existence exponent} $p_0=1$ and the \emph{Fujita exponent} $p_F=1+2/N$, such that for $0<p< p_0$ all the solutions are globally  defined in time, for $p_0<p< p_F$ all the solutions blow up, whereas for $p>p_F$ there exist both, global solutions and blowing-up solutions. The limit cases $p=p_0$ and $p=p_F$ belong, for this problem and respectively, to the global existence range and blow-up range, see also \cite{Hayakawa}. From this result several extensions have been investigated in the subsequent years, for all values of $m>0$, or with different diffusion operators and reactions; we mention the monographs \cite{GalaktionovKurdyumovMikhailovSamarski,QuittnerSouplet,SamarskiGalaktionovKurdyumovMikhailov} for equation \eqref{eq.principal} with $L=\infty$.

In the presence of a localized reaction, $L<\infty$, problem \eqref{eq.principal} has been studied, again  in the context of blow-up, in \cite{BaiZhouZheng,FerreiradePabloVazquez,Pinsky}, where the critical exponents $p_0$ and $p_F$, as defined above, have been obtained in the different situations. In dimension $N=1$ they are
\begin{equation}\label{exp-p0}
p_0=\max\{1,\dfrac{m+1}2\},\qquad
p_F=m+1.
\end{equation}

Our purpose in this work is to study the behaviour of global solutions and to characterize wether they are bounded or not. In the last case, we have that~\eqref{bup} holds with $T=\infty$, a phenomenon that is called {\it grow-up}. Our first task is to prove that if $p>p_0$ all global solution are bounded.

\begin{teo}\label{teo-bounded}
Let $p>p_0$ and $u$ be a global solution of \eqref{eq.principal}. Then, $u$ is bounded.
\end{teo}

Therefore, we focus our attention in the interval $p\le p_0$. Observe that when $L=\infty$, every solution with $p\le p_0=1$ has grow-up and in fact tends to infinity at every point. Next we prove that the same occurs for $L<\infty$ with the new value of $p_0$.

\begin{teo}\label{teo-GUP} If  $0<p\le p_0$ then the solution is defined for all times and $\lim\limits_{t\to\infty}u(x,t)=\infty$ uniformly in compact sets of $\mathbb{R}$.
\end{teo}

\begin{figure}[!ht]
\hspace*{-1.5cm}
\includegraphics[scale=.6]{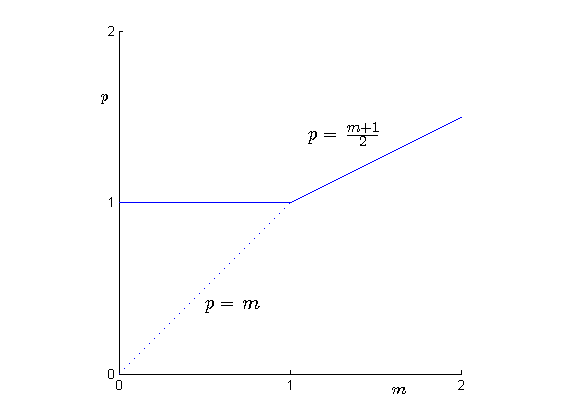}
\caption{Grow-up region.}
\label{fig.gup}
\end{figure}

We now study the grow-up rate, that is, the speed at which the solution goes to infinity. If $L=\infty$ and $0<p<p_0=1$, the grow-up is independent of $m$ as is driven by the ODE $U'=U^p$, thus giving
$$
u(x,t)\sim t^{\frac1{1-p}},\qquad x\in\mathbb{R},
$$
where by the symbol $\sim$ we mean $0<c_1\le t^{-\frac1{1-p}}u(x,t)\le c_2<\infty$. We call this the natural rate. Observe that it is clearly an upper bound for the rate in our case $L<\infty$. We next show that for problem~\eqref{eq.principal} the grow-up rate depends on the sign of $p-m$, and it coincides with the natural grow-up rate only if $m\le p<1$.

\begin{teo}
  \label{teo-rates-1}
Let $u$ be a solution to \eqref{eq.principal} with $0<p< p_0$. For $t\to\infty$
$$
    u(x,t)\sim t^\alpha,\qquad \alpha=\min\{\frac1{1-p},\frac1{m+1-2p}\},
$$
uniformly in compact subsets of $I$.  \end{teo}

For a heuristical explanation of this dependence, which was already observed in~\cite{FerreiradePabloVazquez} for the corresponding blow-up problem, we look at the rescaled function
$$
v(\xi,\tau)=t^{-\alpha} u(x,t),\qquad \xi=xt^{-\beta},\;\tau=\log t.
$$
It is a solution to the equation
\begin{equation}\label{rescaled-eq}
v_\tau=(v^m)_{\xi\xi}+\beta\xi\, v_\xi+e^{\gamma\tau}\mathds{1}_{\{|\xi|<Le^{-\beta\tau}\}}v^p-\alpha v,
\end{equation}
where
$$
\beta=\frac12(\alpha(m-1)+1),\quad \gamma=\alpha(p-1)+1.
$$
The choices $\gamma=0$ when $p>m$ and $\gamma=\beta$ when $p<m$, which give respectively $\alpha=\frac1{1-p}$ or $\alpha=\frac1{m+1-2p}$, imply that the reaction coefficient tends to 1 in the whole $\mathbb{R}$ or to  a multiple of the Dirac distribution at the origin. Thus the grow-up rate must be given by those limit problems. See Section~\ref{sect-guprate} for the details. Observe that when $p=1>m$ or $p=\frac{m+1}2<m$, we cannot perform the previous rescaling, and the argument suggests an exponential grow-up.

\begin{teo}
  \label{teo-rates-2}
Let $u$ be a solution to \eqref{eq.principal} with $p=p_0$.
\begin{enumerate}
    \item If $m<1$ then
$$
u(x,t)\sim e^t
$$
uniformly in compact subsets of $I$.
\item If $m=1$ then
$$
  \lim_{t\to\infty}\frac{\log u(x,t)}{t}=\lambda_0(L),
$$
where $0<\lambda_0(L)<1$,   uniformly in compact subsets of $\mathbb{R}$.
\item If $m>1$, there exists $\beta_*>0$ such that
$$
  \lim_{t\to\infty}\frac{\log u(x,t)}{t}=\beta_* L^2
$$
uniformly in compact subsets of $\mathbb{R}$.
\end{enumerate}
\end{teo}
In the second case, $p=m=1$, if we impose to the initial data a behaviour at infinity like $e^{-\sqrt{\lambda_0}|x|}$, then we can obtain a more precise result, namely
\begin{equation}
    \label{rates3}
    u(x,t)\sim e^{\lambda_0t}.
  \end{equation}

The fact that the reaction is confined to the interval $I$  is observed in the following result, which shows that the grow-up rate is different outside that interval provided $p>m$. We must assume for the initial datum
\begin{equation}
  \label{condicion_dato}
|x|^{2}u_0^{1-m}(x)\sim\left\{
\begin{array}{ll}
1, & p<1,\\
\log x,\qquad &p=1.
\end{array}\right.
\end{equation}

\begin{teo}\label{teo-rates-3} Let $u$ be a solution to \eqref{eq.principal} with $0<p\le p_0$.
\begin{enumerate}
\item If $p\le m$ then the rate $u(x,t)\sim t^{\frac1{m+1-2p}}$ is true in every compact set of $\mathbb{R}$.
\item If $m<p$ and $u_0$ satisfies \eqref{condicion_dato} then for $t\to\infty$ and $x\not\in I$,
$$
 u(x,t)\sim   t^{\frac1{1-m}}.
$$
  \end{enumerate}
\end{teo}

The paper is organized as follows: Section~\ref{sect-bdd-notbb} deals with the question of whether the global solutions are bounded or not; in Section~\ref{sect-guprate} we study the rate at which the unbounded solutions tend to infinity. We include an appendix to comment upon the question of uniqueness of solution.

\section{Grow up}\label{sect-bdd-notbb}\setcounter{equation}{0}

By the definition of global existence exponent $p_0$, if $p<p_0$ all the solutions are global, while if $p>p_0$ there exist either blow-up  and global solutions. The value of $p_0$ is given in \eqref{exp-p0}. We prove in this section that in the lower range $p\le p_0$  the solutions are unbounded, and in fact they go to infinity  at all points. On the other hand, for $p>p_0$ all global solutions are bounded. We begin with this last result.

\begin{proof}[Proof of Theorem~\ref{teo-bounded}]
We know that if $p_0<p\le p_f$ all solution blow-up at finite time, so we are led to consider only the case $p>m+1$. In order to arrive at a contradiction we assume that $u$ is a global unbounded solution to \eqref{eq.principal}. We claim that in this case the
decreasing energy functional associated to problem \eqref{eq.principal}
\begin{equation}\label{eq.energia}
E_u(t)=\frac12\int_\mathbb{R} |(u^m)_x|^2\,dx-\frac{m}{p+m}\int_{-L}^L u^{p+m}\,dx.
\end{equation}
goes to $-\infty$ as $t\to\infty$. On the other hand, by the concavity argument it is easy to see that if $E_u(t_0)\le 0$ for some $t_0\ge0$,
then $u$ blows up in a finite time, see for instance \cite{QuittnerSouplet}. This contradiction gives us that all global solution are bounded.

In order to proof the claim, we follow the ideas of \cite{Giga}.
We first take an increasing sequence $t_n\to \infty$ such that
$$
\lambda_n=
\|u(\cdot,t_n)\|_{L^\infty(\mathbb R)}=\|u\|_{L^\infty(\mathbb R\times [0,t_n])}\,.
$$
In
this way the sequence $\lambda_n$ is nondecreasing. Also, there
exists a sequence $x_n\in I$ such that
\begin{equation}
\frac{\lambda_n}{2}\le u(x_n,t_n)\le \lambda_n\,.
\label{>12}
\end{equation}
Notice that since the reaction is localized in $I$, it is easy to see that there exists a time $t_0$ such that for $t\ge t_0$ the maximum of $u(\cdot,t)$ is reached in $I$. Assume first that, for a subsequence if necessary, $\lim_{n\to\infty} x_n=x_0\in(-L,L)$.

We define the rescaled function
$$
\phi_n(y,s)= \frac{1}{\lambda_n} u(x_n+A_ny,t_n+B_ns)\,,
$$
whit $A_n=\lambda_n^{(m-p)/2}$ and $B_n=\lambda_n^{1-p}$. Since
$p>m+1$, we have that both $A_n$ and $B_n$ go to zero as $n\to
\infty$.
The function $\phi_n$ so defined satisfies the equation
$$
(\phi_n)_s=(\phi_n^m)_{yy}+ \mathds{1}_{I_n}\phi_n^p\,,\qquad \mbox{for }
(y,s)\in\mathbb R\times (-\frac{t_n}{B_n},0)\,,
$$
where $I_n=\{y\in \mathbb{R} : x_n+A_n y\in I\}$.
Observe that the family $I_n$ expand to cover the whole $\mathbb{R}$
as $n\to \infty$. Moreover, $0\le
\phi_n\le 1$. Then, by standard regularity theory, we have that
$\phi_n(y,s)\to\Phi(y,s)$ uniformly in compact sets of
$\mathbb{R}\times (-\infty,0)$, where the function $\Phi$
satisfies the equation
\begin{equation}\label{eq-limite}
\Phi_s=(\Phi^m)_{yy}+\Phi^p\,,\qquad \mbox{in }
\mathbb{R}\times (-\infty,0)\,.
\end{equation}

On the other hand,
$$
\begin{array}{rl}
\displaystyle\int_{s_1}^{s_2}\int_{I_n} |(\phi_n)_t|^2\,dy\,ds=&
\displaystyle\frac{B_n^2}{\lambda_n^2}\int_{s_1}^{s_2}\int_{I_n}
|u_t(x_n+A_n y,t_n+B_n s)|^2\,dy\,ds \\ [.3cm]
\le & \displaystyle
\lambda_n^{- \frac{2+m+p}{2}}\int_0^{t_n}\int_{I} |u_t(x,t)|^2\,dx\,dt
\\ [.3cm]
=&\displaystyle
\lambda_n^{- \frac{2+m+p}{2}} (E_u (0)-E_u (t_n))\,.
\end{array}
$$
This implies that, if $E_u$ is
bounded, then
$$
\displaystyle\int_{s_1}^{s_2}\int_{I_n}
|(\phi_n)_t|^2\,dy\,ds\to 0
$$
for every $0<s_1<s_2<\infty$ and, therefore, the limit function
$\Phi$ does not depend on $s$. Moreover it is nonnegative and
nontrivial since $\Phi(x_0)\ge 1/2$ by (\ref{>12}). We conclude with
a contradiction since it is well know that there does not exist nontrivial nonnegative solution to $(\Phi^m)''+\Phi^p=0$.

In the case $\lim_{n\to\infty}x_n=\pm L$ we get the same conclusion since there exist no solution to $(\Phi^m)''+\mathds{1}_{I_\infty}\Phi^p=0$ neither for $I_\infty=(-\infty,0)$ nor for $I_\infty=(0,\infty)$.

This contradiction proves that the energy functional is unbounded as we wanted to show, and the proof is finished.
\end{proof}

Next we consider the case $p\le p_0$ and show that all solutions are global and unbounded in $\mathbb{R}$.

\begin{proof}[Proof of Theorem~\ref{teo-GUP}] We divide the proof in several steps, showing that: $i)$ $u$ is global, $ii)$ $u$ is unbounded, $iii)$ we may assume $u$ symmetric and nonincreasing for $x>0$,  $iv)$ $u(0,t)\to\infty$, and $v)$ $u(x,t)\to\infty$ for every $x\in\mathbb{R}$.

\textsc{Step $i)$} The fact that $u$ is globally defined is proved in \cite{FerreiradePabloVazquez}
for $1<p\le(m+1)/2$, while for $p\le1$ follows by comparison with the supersolution
$$
\overline u(t)=Me^t,\qquad M=1+\|u_0\|_\infty.
$$

\textsc{Step $ii)$} Assuming  $u(x,t)\le M$ for every $x$ and $t$ we have that $u$ is a supersolution to the equation
$$
v_t= (v^m)_{xx}+M^{p-q}a(x) v^q,
$$
for some $q\in(p_0,p_F)$, which is a contradiction with the fact that $v$ is unbounded.

\textsc{Step $iii)$} Approximating the reaction coefficient $a(x)$ by smooth symmetric, nonincreasing for $x>0$ functions it is easy to see that if $u_0$ is symmetric, nonincreasing for $x>0$, the same is true for $u(\cdot,t)$ for every $t>0$. Thus the next steps follow by comparison with a small initial datum with that properties.

\textsc{Step $iv)$} Assume by contradiction that there exists some sequence of times $t_j\to\infty$ such that
$\lim_{j\to\infty}u(0,t_j)=\Lambda<\infty$. Let us consider the Lyapunov energy functional defined in \eqref{eq.energia}, which is nonincreasing,
$$
E_u'(t)=-\frac{4m}{(m+1)^2} \int_{-\infty}^\infty \left((u^{\frac{m+1}2})_t\right)^2\,dx.
$$
Through our sequence $t_j$ we have that $E_u(t_j)$ is bounded. Therefore, by standard arguments $u$ should converge (up to a new subsequence of times) to an stationary solution. Since the one dimensional Green function is unbounded there is no nontrivial stationary solutions, so we get $\Lambda=0$. Now, the fact that $\limsup_{t\to\infty}u(0,t)=\infty$ implies that we can find another sequence of times $\tau_k\to\infty$ such that $\lim_{k\to\infty}u(0,\tau_k)=1$, and therefore $u$ must converge through this subsequence to a nontrivial stationary solution, which is a contradiction.

\textsc{Step $v)$} We have that given any large $A>0$ it holds $u(0,t)>A$  for every $t>t_A$. We then
use a comparison argument with a special subsolution of Polubarinova-Kochina type~\cite{PolubarinovaKochina}. These are functions  $w_A(x,t)=F_A(xt^{-1/2})$, with profile $F_A(\xi)=AF(A^{\frac{1-m}2}\xi)$, and $F$ satisfying
$$
  \begin{cases}
    (F^m)''+\dfrac12\xi F'\ge0, & \xi>0,\\
    F(0)=1,\;F(\infty)=0.
  \end{cases}
$$
The function $w_A$ is a subsolution to the pure diffusion equation, thus also a subsolution to our equation. If $m\ge1$  there is  in fact a solution, see~\cite{PolubarinovaKochina}, explicit if $m=1$. In the case $0<m<1$ we see that $F^m(\xi)=(1-B\sqrt\xi)_+$ do the job provided $B>0$ is large.

Now we observe  that  $w_A(0,t-t_A)=A<u(0,t)$ for every $t>t_A$,  $w_A(x,0)=0\le u(0,t_A)$ for every $x>0$ and $\lim\limits_{t\to\infty}w_A(x,t)=A$ for every $x\ge0$. This implies  $u(x,t)\ge w_A(|x|,t-t_A)\to A$ for every $x\in\mathbb{R}$, and $A$ is arbitrary.
\end{proof}

\section{Grow-up rate}\label{sect-guprate}\setcounter{equation}{0}

This section is devoted to  study  the speed at which the  solutions to problem~\eqref{eq.principal} tend to infinity when $0<p\le p_0$. We prove here Theorems~\ref{teo-rates-1}, \ref{teo-rates-2} and \ref{teo-rates-3}, divided in different subsections according to the different regions of the parameters.

An easy upper estimate of the grow-up rate is given by comparison with the solutions of the ODE
$$
U'(t)=U^p(t).
$$
This gives,
\begin{equation}
  \label{flat}
  u(x,t)\le\begin{cases}  (M^{1-p}+(1-p)t)^{\frac1{1-p}},& \quad \text{if } p<1, \\
  Me^t,& \quad \text{if } p=1
  \end{cases}
\end{equation}
where $M=\|u_0\|_\infty$.

In the case of global reaction it is proved in~\cite{AguirreEscobedo,dePabloVazquez} that this is indeed the grow-up rate when $0<p<1$, that is
$$
u(x,t)\sim t^{\frac1{1-p}}
$$
for $t$ large, uniformly in $\mathbb{R}$.

If $p=1$ and $m=1$  it is $u(\cdot,t)\sim t^{-1/2} e^t$. If $m\neq1$ we can perform the change of variables
$$
v(x,\tau)=e^{-t}u(x,t),\quad \tau=\frac{e^{(m-1)t}-1}{m-1},
$$
so that $v$ satisfies $v_\tau=(v^m)_{xx}$. Thus, if $m>1$ the decay $v(\cdot,\tau)\sim\tau^{-\frac{1}{m+1}}$ as $\tau\to\infty$ (see for instance \cite{V1}) implies $u(\cdot,t)\sim e^{\frac{2}{m+1}t}$ as $t\to\infty$, always uniformly in compact sets. If $0<m<1$, since $\lim_{t\to\infty}\tau=\frac1{1-m}$ we get $u(\cdot,t)\sim e^t$.
In any case
$$
\lim_{t\to\infty}\frac{\log u(x,t)}t= \min\{1,\frac2{m+1}\}
$$
uniformly in compact sets of $\mathbb{R}$.

We see next that this estimates are far from being sharp in some cases.

\subsection{Linear diffusion $m=1$, linear reaction $p=1$}

\

In this case we  have that the presence of a localized reaction provokes a growth of the solutions that is strictly slower than that of the solutions with global reaction, that is, we prove that the solutions behave for large times like an exponential, but the exponent depends on the length $L$ and is strictly less than 1.

\begin{lema}
  \label{lema-rates-2}
Let $u$ be a solution to problem \eqref{eq.principal} with  $m=p=1$. There exists an increasing function $\lambda_0:(0,\infty)\to(0,1)$ such that
$$
  \lim_{t\to\infty}\frac{\log u(x,t)}{t}=\lambda_0(L),
$$
uniformly in compact subsets of $\mathbb{R}$.
\end{lema}
\begin{proof}
We get the result by comparison with  explicit  growing-up subsolutions and supersolutions of the form
$$
w(x,t)=e^{\lambda t}\varphi_\lambda(|x|).
$$
The profile $\varphi=\varphi_\lambda$ satisfies
\begin{equation*}\label{two-bessel}
\left\{
\begin{array}{ll}
\varphi''+(1-\lambda)\varphi=0,&\quad\text{if } 0<r<L,\\
\varphi''-\lambda\varphi=0,&\quad\text{if } r>L,\\
\varphi(r)>0,&\quad\text{for } r\ge0,\\
\varphi'(0)=0,
\end{array}
\right.
\end{equation*}
We get
$$
\varphi(r)=\left\{
\begin{array}{ll}
\cos(\sqrt{1-\lambda}\,r) \qquad & \text{if } r<L,\\
C_1e^{\sqrt{\lambda}\,r}+C_2e^{-\sqrt{\lambda}\,r} & \text{if } r>L.
\end{array}\right.
$$
Notice that if  $L\sqrt{1-\lambda}\ge\pi/2$ the solution changes sign in $(0,L]$. So we consider $L\sqrt{1-\lambda}<\pi/2$. The compatibility condition at $r=L$ means
$$
\left\{
\begin{array}{l}
\displaystyle C_1=\frac{e^{-\sqrt\lambda L}}{2\sqrt\lambda}\left(\sqrt\lambda \cos(L\sqrt{1-\lambda})-\sqrt{1-\lambda}\sin(L\sqrt{1-\lambda})\right), \\ [4mm]
\displaystyle C_2=\frac{e^{\sqrt\lambda L}}{2\sqrt\lambda}\left(\sqrt\lambda \cos(L\sqrt{1-\lambda})+\sqrt{1-\lambda}\sin(L\sqrt{1-\lambda})\right).
\end{array}
\right.
$$
We immediately deduce $C_2>0$. On the other hand, the sign of $C_1$ coincides with the sign of $\lambda-\lambda_0$, where $\lambda_0=\lambda_0(L)$ is the unique root of the function
$$
h(\lambda,L)=\sqrt\lambda \cos(L\sqrt{1-\lambda})-\sqrt{1-\lambda}\sin(L\sqrt{1-\lambda}).
$$
Remember that $L\sqrt{1-\lambda}<\pi/2$. We  have $\lambda_0(L)$ increasing, with $\lambda_0(0)=0$ and $\lim\limits_{L\to\infty}\lambda_0(L)=1$.% , $\lambda_0\in(1-\frac{4\pi^2}{L^2},1)$ for every $L>\pi/2$, and thus $\lim\limits_{L\to\infty}\lambda_0(L)=1$.

Taking $\lambda>\lambda_0$ we get a positive unbounded profile $\varphi_\lambda$, while for $\lambda<\lambda_0$ we obtain a subsolution by the procedure of truncation by zero.

Let now $u$ be a solution to problem~\eqref{eq.principal}.
We first compare from above. For every $\varepsilon>0$ let
$$k_\varepsilon=\frac{\|u_0\|_\infty}{\min_{r>0} \varphi_{\lambda_0+\varepsilon}(r)},$$
so that
$$u_0(x)\le k_\varepsilon\varphi_{\lambda_0+\varepsilon}(x).$$
Therefore
$$
u(x,t)\le k_\varepsilon e^{(\lambda_0+\varepsilon)t}\varphi_{\lambda_0+\varepsilon}(x).
$$
Taking logarithms  we get
$$\lim_{t\to\infty}\frac{\log u(x,t)}{t}\le\lambda_0+\varepsilon$$
for  $|x|<R$, for every $R>0$ and every $\varepsilon>0$.

On the other hand, let $R_\varepsilon$ be the point where $\varphi_{\lambda_0-\varepsilon}$ vanishes, and assume without loss of generality that $\mu_\varepsilon=\min_{|x|<R_\varepsilon}u_0(x)>0$. We thus get
$$
u(x,t)\ge \mu_\varepsilon e^{(\lambda_0-\varepsilon)t}\varphi_{\lambda_0-\varepsilon}(x)
$$
for every $|x|<R_\varepsilon$. We conclude as before observing that $R_\varepsilon\to\infty$ as $\varepsilon\to0$.

\end{proof}

\subsection{The supercritical case $m\le p\le 1$, $m<1$}

\
We show here that the natural grow-up rates \eqref{flat}  are sharp in $I$ by proving the  lower bound. To do that we compare with a subsolution in separated variables with compact support,
$$
w(x,t)=\psi(t)\varphi(|x|).
$$
Notice that since $u$ has global grow-up, $u(x,t+t_0)\ge w(x,0)$ for $t_0$ large enough, then the comparison of the initial data is granted by a time shift. Thus the behaviour of $u$ is given by the behaviour of $\psi$. We consider different cases.

\begin{lema}
Let $p=m<1$ and let $u$ be a growing-up solution of \eqref{eq.principal}. Then
$$
u(x,t)\ge ct^{\frac1{1-m}}
$$
uniformly in compact sets of $\mathbb{R}$.
\end{lema}
\begin{proof}
In this case, $\psi=\psi_\lambda$ satisfies $\psi'=\lambda\psi^m$, and $\varphi=\varphi_\lambda$ is a solution to
$$
\left\{
\begin{array}{ll}
( \varphi^m)''+a(r)\varphi^m-\lambda \varphi=0, \quad & r>0,\\
\varphi(0)=1,\ (\varphi^m)'(0)=0.
\end{array}\right.
$$
It is easy to check that  there exists a limit value $\lambda^*>0$ such that for every $0<\lambda<\lambda^*$ the solution $\varphi$ is positive and decreasing in $[0,R_\lambda)$, with $\varphi(R_\lambda)=0$, and $\lim_{\lambda\to\lambda^*}R_\lambda=\infty$.
In fact $\varphi$ crosses the axis at some point if $\lambda=0$, so by continuous dependence with respect to $\lambda$ the same holds  when $\lambda$ is small.
On the other hand, the solution corresponding to $\lambda=1$ satisfies  $\varphi(r)=1$ in $0<r<L$ and it increases to infinity for $r>1$. The existence of $\lambda^*$ is then standard.

Let now $x_0\in\mathbb{R}$ be any point. We take $\lambda\sim\lambda^*$ so that $R_\lambda>|x_0|$. Comparison in $[0,R_\lambda]$ gives the grow-up rate in the interval $(-|x_0|,|x_0|)$.
\end{proof}

\begin{lema}\label{lem.bola}
Let $m<p\le1$ and let $u$ be a growing-up solution of \eqref{eq.principal}. Then for every $|x|<L$ it holds
$$
u(x,t)\ge c \left\{
\begin{array}{ll}
t^{\frac1{1-p}}\quad &\text{if}\quad p<1,\\
e^t & \text{if}\quad p=1.
\end{array}\right.
$$
\end{lema}
\begin{proof}
Let $\phi$ be a solution to the problem
$$
\left\{
\begin{array}{l}
(\phi^m)''+ \phi^p=0,\\
\phi'(0)=0,\\
\phi(0)=1.
\end{array}\right.
$$
Since $\phi^m$ is concave, $\phi$ vanish at some point $R_0<\infty$. Now, we consider
$$
\varphi(x)=A \phi(A^{(p-m)/2}x),\qquad A=\left(\frac{R_0}{L}\right)^{\frac{2}{p-m}},
$$
which satisfies the same equation and vanish at $x=L$. The function $\psi$ is defined as
$$
\psi'=\frac{\psi^p-\psi^m}{\varphi^{1-p}(0)},\qquad \psi(0)>1.
$$
Then, $\psi'\sim\psi^p$ as $t\to\infty$.
\end{proof}

In order to obtain the grow-up rate outside $I$, we consider  solutions in self-similar form to the pure diffusion equation
\begin{equation}\label{pme}
U_t=(U^m)_{xx}
\end{equation}
for $x\ne0$.
According to the behaviour of the solution to problem~\eqref{eq.principal} in $I$, the self-similar solutions used are different depending on the reaction exponent being $p<1$ or $p=1$. That is, we consider
\begin{equation}\label{SS-PME}
U(x,t)=\left\{
\begin{array}{lll}
t^{\alpha} f(\xi)\qquad &\xi=|x| t^{\beta}\qquad &\mbox{if } m<p<1\\
e^{\alpha t} f(\xi)&\xi=|x| e^{\beta t} & \mbox{if } m<p=1
\end{array}\right.
\end{equation}
Actually $\alpha=1/(1-p)$ if $p<1$ and $\alpha=1$ for $p=1$. In both cases the profile $f$ verifies the equation
\begin{equation}\label{perfil-pme}
(f^m)'' =\alpha f+\beta \xi f',\quad\xi>0,
\end{equation}
where $f'$ denotes $df/d\xi$. The similarity exponents satisfy the relations
\begin{equation}\label{exponents-SS}
\begin{array}{ll}
(1-m)\alpha=2\beta+1,&\quad\text{if } p<1, \\
(1-m)\alpha=2\beta ,&\quad\text{if } p=1,
\end{array}
\end{equation}
which in particular implies $\beta=\frac{(p-m)\alpha}{2}>0$.
Observe that if $f$ is a solution to \eqref{perfil-pme}, then $f_\lambda(\xi)=\lambda^{\frac2{1-m}}f(\lambda\xi)$ is also a solution, so we obtain a uniparametric family of solutions in terms of the value at the origin.

Equation~\eqref{perfil-pme} has been studied in detail in several papers but always with $m>1$, cf.~\cite{GildingPeletier,GildingPeletier2,Gilding}.

\begin{lema} Let $0<m<1$ and $\alpha,\,\beta>0$ be three  parameters such that $ (1-m)\alpha-2\beta\in\{0,1\}$. Then there exists a unique bounded, non-negative decreasing self-similar profile $f$ satisfying $f(0)=1$ and
$$(f^m)'' =\alpha f+\beta\xi f'\ge 0,\quad \xi>0.$$
Moreover, the behaviour of $f$ as $\xi\to\infty$ is given by
$$
f(\xi)\sim
\left\{
\begin{array}{ll}
\xi^{\frac{-2}{1-m}}\quad & \mbox{if }(1-m)\alpha-2\beta=1,\\
\xi^{\frac{-2}{1-m}} (\log\xi)^{\frac1{1-m}} & \mbox{if }(1-m)\alpha-2\beta=0.
\end{array}\right.
$$
\end{lema}
\begin{proof}
We introduce the following variables
$$
X=\frac{\xi f'}{f}, \qquad Y=\frac1{m} \xi^2 f^{1-m}, \qquad \eta=\log\xi.
$$
The resulting system is
$$
\left\{
\begin{array}{l}
\dot{X}=X-mX^2+Y(\alpha+\beta X)\\
\dot{Y}=(2+(1-m)X)Y
\end{array}\right.
$$
where $\dot{X}=dX/d\eta$. We look for non-negative decreasing profiles, so we consider only the second quadrant $X<0\,,\, Y>0$.

First we consider the case $(1-m)\alpha=2\beta+1$. The critical points are
$$
A=(0,0), \qquad B=\left(\frac{1}{m},0\right), \qquad C=\left(\frac{-2}{1-m},2\frac{1+m}{1-m}\right).
$$
The local analysis around those points is straightforward: $A$ is a repeller and $B$ and $C$ are saddle points. Let us define the curves in the phase space,
$$
\begin{array}{l}
\displaystyle \Gamma_1=\left\{\frac{-2}{1-m}\le X\le 0\,,\, Y=0\right\},\\[3mm]
\displaystyle \Gamma_2=\left\{X=\frac{-2}{1-m}\,,\, 0\le Y\le 2\frac{1+m}{1-m}\right\},\\[3mm]
\displaystyle \Gamma_3=\left\{\frac{-2}{1-m}\le X\le 0\,,\, Y= \frac{mX^2-X}{\alpha+\beta X}\right\}.
\end{array}
$$
Note that at $\Gamma_1\cup\Gamma_2$ we have $\dot{X}\le 0$ and $\dot{Y}=0$,
while for $\Gamma_3$  we have $\dot{X}= 0$ and $\dot{Y}>0$. Then, if we look at the orbits backward in time, the region surrounded by those curves,
$$
\Omega=\left\{\frac{-2}{1-m}\le X\le 0\,,\, 0\le Y\le \frac{mX^2-X}{\alpha+\beta X}\right\}
$$
is invariant. Even more, in this region $\dot{Y}>0$ so, the only possibility for the orbit passing through a point in either $\Gamma_2$ or $\Gamma_3$  is that it comes from  $A$. Therefore there exists a separatrix  connecting the points $A$ and $C$.

\begin{figure}[!ht]
\hspace*{-1.5cm}
\includegraphics[scale=.5]{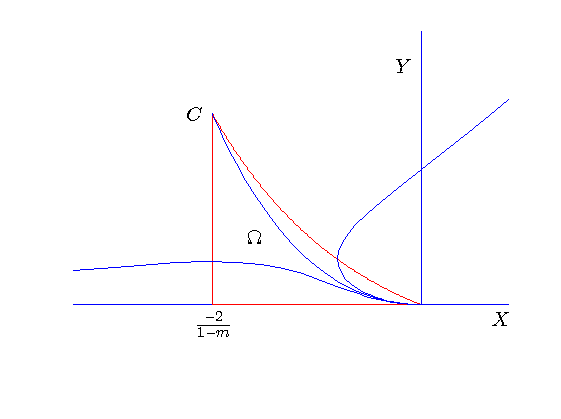}
\vspace{-1cm}
\caption{The phase plane for $(1-m)\alpha=2\beta+1$.}
\label{fig.gup}
\end{figure}

This  orbit gives us a decreasing positive self-similar profile such that
$$
f_*(\xi) \sim \xi^{-2/(1-m)} \quad \mbox{as } \xi\to\infty.
$$
Also, for $\xi$ near zero we have that $Y\sim e^{2\eta}$, and therefore $f_*(\xi)\sim 1$, so we may fix the value $f_*(0)$.

\

Now let $(1-m)\alpha =2\beta$. In this case the critical point $C$ no longer exist. However we can use the same argument as before to obtain a separatrix  which connects the point $A$ with the point at infinity with $X=-2/(1-m)$. Therefore we get a decreasing positive self-similar profile with the same behaviour as before near the origin, but satisfying for $\xi$ large
$$
f_*(\xi)\sim \left(\frac{\log\xi}{\xi^2}\right)^{\frac1{1-m}}.
$$

\begin{figure}[!ht]
\hspace*{-1.5cm}
\includegraphics[scale=.5]{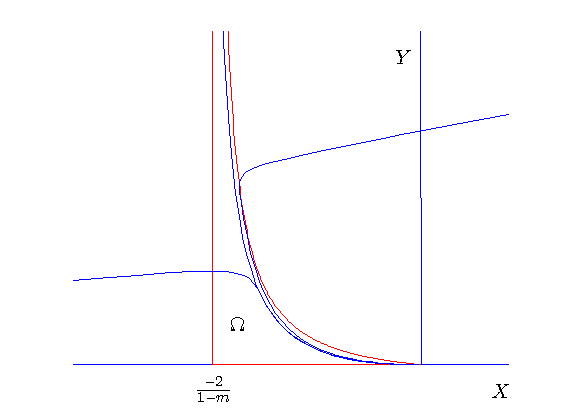}
\caption{The phase plane for $(1-m)\alpha=2\beta$.}
\label{fig.gup}
\end{figure}

Finally, we observe that in both cases the separatrix  lives in the set $\Omega$, which implies  $X\ge -2/(1-m)$. This implies
$$
\alpha f_*+\beta \xi f_*'=f_* (\alpha+\beta X)\ge 0.
$$
\end{proof}

\begin{cor}
There exists a uniparametric family of symmetric self-similar solutions $U$ to equation~\eqref{pme} for $x\ne0$ in the form \eqref{SS-PME}, \eqref{exponents-SS}, which is decreasing for $x>0$ and increasing for $t>0$. Moreover, for $|x|$ large
$$
|x|^{2}U^{1-m}(x,t)\sim
\left\{
\begin{array}{ll}
t, & p<1,\\
\log |x|+\frac{1-m}2 t,\quad &p=1.
\end{array}\right.
$$
Each solution is characterized by the value of the profile at the origin.
\end{cor}

\begin{lema}
Let $m<p\le 1$ and $u_0(x)$ satisfying \eqref{condicion_dato}. Then, for every $|x|>L$ and $t$ large it holds
$$
u(x,t)\sim
t^{\frac1{1-m}}.
$$
\end{lema}
\begin{proof}
We first consider  the case $p<1$. By Lemma~\ref{lem.bola} we know that $u(0,t)\ge \overline c t^{1/(1-p)}$ for $t\ge t_0$. On the other hand, for $t\in[0,t_0]$ we note that $u$ is a supersolution of the fast diffusion equation, then $u(0,t)\ge c$. Thus, there exists a constant $\overline c$ such that $u(0,t)\ge \overline c t^{1/(1-p)}$ for $t\ge 0$. Then,  $u$ is a supersolution to
$$
\left\{
\begin{array}{ll}
w_t=(w^m)_{xx}, \qquad &|x|>0,\, t>0,\\
w(0,t)=\overline c t^{1/(1-p)},\\
w(x,0)=u_0(x).
\end{array}\right.
$$
Let us consider $\underline w(x,t)=AU(x,t)$, where $U$ is the selfsimlar solution constructed before with $\alpha={1/(1-p)}$, $\beta=\alpha(p-m)/2$ and $f(0)=1$.

Since $U_t\ge 0$ we get
$$
\underline w_t-(\underline w^m)_{xx}=(A-A^m)U_t\le 0
$$
provided $A<1$. At time $t=0$ we have that $\underline w(x,0)=0$, because the profile $f$ is bounded. Finally, $\underline w(0,t)=A t^\alpha f(0)\le \overline c t^\alpha$ for $A$ small enough. Then by comparison $u\ge \underline w$, and thus, for $t$ large,
$$
u(x,t)\ge \underline w(x,t)\sim |x|^{\frac{-2}{1-m}} t^{\frac{1}{1-m}}.
$$

In order to obtain the upper estimate, we observe that from \eqref{flat},  $u(L,t)\le (M^{1-p}+(1-p)t)^{\frac1{1-p}}$ for $t\ge0$. Then, for $t_0$ large enough, $u$ is a subsolution to
$$
\left\{
\begin{array}{ll}
w_t=(w^m)_{xx}, \qquad &|x|>L,\, t>0,\\
w(L,t)=\underline c (t+t_0)^{1/(1-p)},\\
w(x,0)= u(x,0).
\end{array}\right.
$$
Consider now the function $\overline w(x,t)=AU(x-L,t+t_0)$. Thanks to the hypothesis on the decay of the initial data, cf. \eqref{condicion_dato},  we get that $\overline w(x,t)$ is a supersolution of the above problem. Then, by comparison
$$
u(x,t)\le \overline w(x,t)\sim |x|^{\frac{-2}{1-m}} t^{\frac{1}{1-m}},
$$
for $t$ large.

Finally for $p=1$, we argue in the same way, taking $\underline w=AU(x,t+1)$ and $\overline w=AU(x,t)$, where  $U(x,t)$ be the selfsimlar solution constructed before with $\alpha=1$,  $\beta=(1-m)/2$ and $f(0)=1$.
\end{proof}

\subsection{The subcritical case $p<m$}

\

We construct a grow-up sub and super-solution. The function obtained in this case is not of self-similar form but the matching of a self-similar function with a grow-up parabola. Consider the even $C^1$ function obtained by reflection from
$$
w(x,t)= \left\{
\begin{array}{ll}
\Big(V^m(0,t)+\frac{K}{2L}V^p(0,t)(L^2-x^2)\Big)^{1/m},\qquad &0\le x\le L,\\
V(x-L,t),&  x\ge L,
\end{array}\right.
$$
where $V$ is a self-similar sub or super-solution to the problem
$$
\left\{
\begin{array}{ll}
V_t=(V^m)_{xx}, & x>0,\, t>0,\\
-(V^m)_x(0)= K V^p(0), \qquad & t>0.
\end{array}\right.
$$
We look, analogously as before, for potential self-similarity if $m<p<p_0$, while for $p=p_0>1$ we consider exponential self-similarity, that is, we put
$$
V(x,t)=\left\{
\begin{array}{ll}
t^\alpha F(x t^{\beta}), \quad & \text{ if } m<p<p_0,\\
e^{\lambda t} F(x e^{\mu t}), & \text{ if } p=p_0>1.
\end{array}\right.
$$

\begin{lema}
Let $p<\min\{m,p_0\}$. Then,
$$
u(x,t)\sim t^{\frac{1}{m+1-2p}}
$$
uniformly in compacts of $\mathbb R$.
\end{lema}
\begin{proof}
To obtain the lower estimate, we consider
$$
V=t^\alpha (A-B x t^{(m-p)\alpha})_+^\gamma
$$
where $\alpha=1/(m+1-2p)$, and $\gamma=2/m$ if $m\le1$, $\gamma=1/(m-1)$ when $m>1$. The positive constants $A$ and $B$ have to be suitably chosen. In fact, the profile $F(\xi)$, $\xi=x t^{(m-p)\alpha}$ satisfies
$$
\begin{cases}
    \alpha(F-(m-p)\xi F')\le(F^m)'',&\quad \xi>0, \\ -(F^m)'(0)= K F^p(0).
  \end{cases}
$$
provided $A$ is small and $B=\frac{K}{\gamma} A^{\gamma(p-1)-1}$. Therefore $w$ is a subsolution to our equation for $|x|\ge L$.
For $0\le x\le L$ and $t$ large enough it is easy to see that
\begin{equation}
  \label{eq.compor.w}
w_t-(w^m)_{xx}-w^p\sim Ct^{\alpha-1}-(1-\frac{K}{L}) t^{\alpha p}.
\end{equation}
Thus, taking $K<L$ and $t>t_0$, we obtain that $w$ is also a  subsolution in $|x|\le L$. On the other hand, since our solution $u$ has global grow-up, there exists a time $t_1\ge t_0$ such that $u(x,t_1)\ge w(x,t_0)$. Then, by comparison we obtain the lower grow-up estimate.

For the upper grow-up estimate, we consider a self-similar function
$$
V(x,t)=t^\alpha F(\xi), \qquad \xi=x t^{-(m-p)\alpha},
$$
with $F$  positive and unbounded. The existence of  profiles satisfying
$$
\left\{
\begin{array}{l}
(F^m)_{xx}-\alpha F+(m-p)\alpha \xi F'=0,\quad\xi>0, \\
-(F^m)_x(0)=KF^p(0),
\end{array}\right.
$$
follows by considering the energy associated, which is given by
$$
E(\xi)=\frac12\left|(F^m)'\right|^2-\frac{\alpha m}{m+1}F^{m+1}.
$$
Notice that it is non-increasing,
$$
E'(\xi)=-\frac{4m\beta}{(m+1)^2}\left|(F^{\frac{m+1}2})'\right|^2\le0.
$$
Thus, if $F(0)>\left(\frac{K^2(m+1)}{2\alpha m}\right)^\alpha$ we have $E(0)<0$, which implies that $w$ must have a minimum at some finite point $\xi_0$, with
$$
F(\xi_0)\ge F(0)\left(1-\frac{2K^2(m+1)}{\alpha m}(F(0))^{2p-m-1}\right)^{\frac1{m+1}}.
$$
Then, $w$ is a solution of our equation in $|x|\ge L$. For $|x|\le L$ the behaviour \eqref{eq.compor.w} gives us that if $t>t_0$ and $K>L$ then $w$ is supersolution.
Moreover, since $w$ has global grow-up, then for $t_1>t_0$ large $u_0(x)\le w(x,t_1)$ and by comparison we obtain the upper grow-up estimate.
\end{proof}

\begin{lema}
Let $p=p_0>1$. Then there exists a positive constant $\beta_*$ such that for all $\varepsilon>0$
$$
\lim_{t\to\infty}\frac{\log u(x,t)}{t}=\beta_* L^2,
$$
uniformly in compact sets.
\end{lema}
\begin{proof}
Let $V$ be a exponential self-similarity function,
$$
V(x,t)=e^{\lambda t} F(\xi),\qquad \xi=|x| e^{-(m-p)\lambda t},
$$
where the profile $F$ satisfies
$$
\left\{
\begin{array}{l}
(F^m)''-\lambda F+(m-p)\lambda \xi F'=0,\quad\xi>0,\\
-(F^m)'(0)=KF^p(0).
\end{array}\right.
$$
Let $A=F(0)$. Performing the rescaling $F(\xi)=AG(KA^{p-m}\xi)$, we get
$$
\left\{
\begin{array}{l}
(G^m)''-\beta G+(m-p)\beta \eta G'=0,\quad\xi>0,\\
-(G^m)'(0)=G^p(0)=1,
\end{array}\right.
$$
where $\beta=\lambda/K^2$ and $\eta=KA^{p-m}\xi$. In this case the non-increasing energy is given by
$$
E_G(\xi)=\frac12\left|(G^m)'\right|^2-\frac{\beta m}{m+1}F^{m+1}.
$$
Thus, for $\beta$ large the profile has a positive minimum point and it becomes unbounded. On the other hand, for $\beta=0$ the solution changes sing, then by continuous dependence on the parameter for  $\beta$ small enough the profile changes sing.

Therefore, by a soothing argument there exists $\beta_*>0$ such that for $\beta>\beta_*$ the profile is positive and unbounded and for $\beta=\beta_*$ the profile is nonnegative and decreasing. Even more, since $E_{G_*}(\eta)\ge0$ we have
that
$$
|(G_*^{\frac{m-1}{2}})'|^2(\eta)\ge C_m \beta_*.
$$
Thus the profile $G_*$ has compact support.

\begin{figure}[!ht]
\includegraphics[scale=.5]{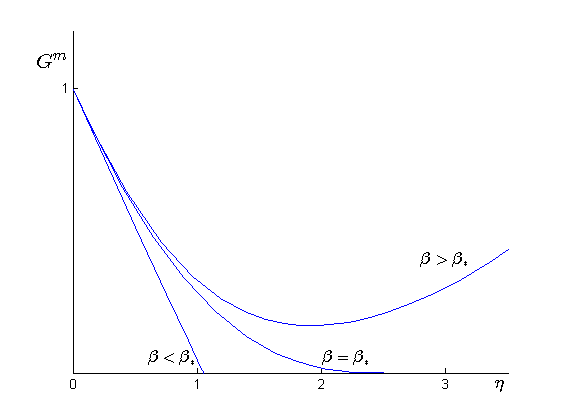}
\caption{The different profiles in terms of $\beta$.}
\label{fig.gup}
\end{figure}

In terms of the function $w$ we have that for $|x|>L$ it is a solution of our problem and it is positive and unbounded for $\lambda>\beta_* k^2$, while for $\lambda=\beta_* k^2$ it is non-negative and   compactly supported. On the other hand, for $0\le x\le L$ and $t$ large enough we have
$$
w_t-(w^m)_{xx}-w^p\sim  e^{\lambda t}+(\frac{K}{L}-1)e^{\lambda p t}.
$$
Then $w$ is a subsolution if $K<L$ and a supersolution if $K>L$.

Summing-up,  we have obtained that for $\alpha>\beta_* L^2$ there exists $K>L$ and $t_0>0$ such that $w$ is a supersolution of our equation for $t\ge t_0$ and $w(x,t_0)\ge \|u_0\|_\infty$. Then by comparison $w(x,t+t_0)\ge u(x,t)$. Since $t_0\to\infty$ as $\alpha\to\beta_* L^2$ we get that
$$
\lim_{t\to\infty} \frac{\log u(x,t)}{t}\le \beta_* L^2.
$$

On the other hand, if  $\alpha<\beta_* L^2$ we can take $K=\sqrt{\alpha/\beta_*}<L$  and $t_0>0$  large enough in such a way that $w$ is a subsolution to our equation for $t\ge t_0$, which is moreover compactly supported. Now we use the fact that our solution $u$ has global grow-up to obtain that there exists $t_1>t_0$ such that $w(x,t_0)<u(x,t_1)$. The lower grow-up rate is obtained again by comparison.
\end{proof}

\section*{Appendix: uniqueness}\label{sect-uniqueness}\setcounter{equation}{0}

As we have said in the introduction, it is easy to obtain the existence of a local in time solution to problem~\eqref{eq.principal}, and the existence time can be prolonged whenever the solution is bounded. It is unique if the reaction is a Lipschitz function of $u$, that is, when $p\ge1$. Uniqueness in the case $p<1$ with global reaction, $L=\infty$, has been studied in~\cite{AguirreEscobedo} when $m=1$, and in~\cite{dePabloVazquez} when $m\ne1$. We characterize now the  uniqueness of solution when $L<\infty$  following those works.

To that purpose let $p<1$ be fixed. We observe that by approximation we can construct a minimal solution $\underline u$ and a maximal solution $\overline u$. By comparison with the flat supersolution given in \eqref{flat} those solutions are globally defined in time.

\begin{teo} Assume $u_0\not\equiv0$ and $0<p<1$.
  \begin{enumerate}
    \item If $u_0(x)\ge\delta>0$ for every $|x|\le L$ then $\underline u\equiv \overline u$, that is, the solution to problem~\eqref{eq.principal} is unique.
        \item Assume $m>1$; if $u_0(x)=0$ for every $|x|\le L'$, for some $L'>L$ then there exists a  time $t_0>0$ such that for every $|x|\le L$, $0<t<t_0$ it is $\underline u(x,t)=0<\overline u(x,t)$.
            \item Assume $m+p\ge2$; if there exist $x_0$ with $|x_0|<L$ and $\varepsilon>0$ small such that  $u_0(x)=0$ for $|x-x_0|<\varepsilon$ then there exists a time $t_0>0$ such that for every $|x-x_0|<\varepsilon/2$, $0<t<t_0$ it is $\underline u(x,t)=0<\overline u(x,t)$.
  \end{enumerate}
\end{teo}

\begin{proof}
  \begin{enumerate}
    \item Uniqueness in the nondegenerate case is trivial since the reaction is Lipschitz.
    \item Using as a subsolution $w(x,t)=ct^\alpha\varphi_1(x)$, where $\varphi_1$ is the first eigenfunction of the Laplacian in $I$ we see that the maximal solution is positive  in $I$ for every positive time. On the other hand, as the Porous Medium Equation possess the finite propagation property, it takes a positive time for the solution to get from $I_{L'}^c$ to $I$.
    \item Using \cite[Lemma 4.1]{dePabloVazquez}, the minimal solution has also finite speed of propagation provided $m+p\ge2$.
  \end{enumerate}
\end{proof}

It would be interesting to prove uniqueness if $\text{supp}(u_0)\equiv I$ or if $\text{supp}(u_0)\cap I\neq\emptyset$ in  the case $m+p<2$ .

\section*{Acknowledgments}

Work supported by the Spanish project  MTM2014-53037-P.

\end{document}